\documentclass{article}
\usepackage[utf8]{inputenc}
\usepackage{amsthm}
\usepackage{amsmath}
\usepackage{amsfonts}
\usepackage{latexsym}
\usepackage{geometry}
\usepackage{enumerate}
\usepackage{csquotes}
\usepackage{graphicx}
\usepackage{comment}
\usepackage{appendix}
\usepackage{hyperref}
\usepackage{bm}

\emergencystretch=\maxdimen
\def \N {\mathcal{N}}
\def \PP {\mathcal{P}}
\def \XX {\mathcal{X}}

\def \dist {{\rm dist}}
\DeclareMathOperator{\Rm}{Rm}
\DeclareMathOperator{\Ric}{Ric}

\DeclareMathOperator{\tr}{tr}
\makeatletter
\newcommand*{\rom}[1]{\rm {\expandafter\@slowromancap\romannumeral #1@}}
\makeatother

\def \tg {\tilde{g}}
\def \bg {\overline{g}}
\def \tf {\tilde{f}}
\def \bf {\overline{f}}
\def \whg {\widehat{g}}
\def \R {\mathfrak{R}}

\def \div {{\rm div}}
\def \EE {\mathcal{E}}

\numberwithin{equation}{section}
\newtheorem{Theorem}{Theorem}[section]
\newtheorem{Proposition}[Theorem]{Proposition}
\newtheorem{Lemma}[Theorem]{Lemma}
\newtheorem{Corollary}[Theorem]{Corollary}
\theoremstyle{definition}
\newtheorem{Definition}[Theorem]{Definition}

\title{The rate of $\mathbb{F}$-convergence for Ricci flows with closed and smooth tangent flows \\ \smallskip 
\emph{\footnotesize In dedication to the $70$th birthday of Professor Peter Li}}
\author{Pak-Yeung Chan, Zilu Ma, and Yongjia Zhang}

\begin{document}

\maketitle

\begin{abstract}
    This article is a continuation of \cite{CMZ21c}, where we proved that a Ricci flow with a closed and smooth tangent flow has unique tangent flow, and its corresponding forward or backward modified Ricci flow converges in the rate of $t^{-\beta}$ for some $\beta>0$. In this article, we calculate the corresponding $\mathbb{F}$-convergence rate: after being scaled by a factor $\lambda>0$, a Ricci flow with closed and smooth tangent flow is $|\log \lambda|^{-\theta}$ close to its tangent flow in the $\mathbb{F}$-sense, where $\theta$ is a positive number, $\lambda\gg 1$ in the blow-up case, and $\lambda\ll 1$ in the blow-down case.
\end{abstract}

\section{Introduction}

In \cite{Bam20a, Bam20b, Bam20c}, the newly proposed and studied notion of $\mathbb{F}$-convergence has greatly expanded the horizon of the field of Ricci flow. Bamler's notion of $\mathbb{F}$-distance can be regarded as the flow version of the Gromov-$W_1$-Wasserstein distance studied by Sturm \cite{St06a,St06b}; the latter measures the closeness of two metric spaces endowed with probability measures, whereas the former is a distance between two metric flows (see \cite{Bam20b}). As an effective version of weak convergence, the success of the $\mathbb{F}$-convergence reveals that the Ricci flow and the Cheeger-Colding theory \cite{CC97,CC20a,CC20b} are united on the fundamental level, since manifolds with Ricci curvature bounded from below can be regarded as super Ricci flows, for which Bamler's $\mathbb{F}$-compactness theorems \cite{Bam20b} are applicable.

One of the central problems in the field of Cheeger-Colding theory is that of tangent cones, and, in particular, of their uniqueness. For instance, Cheeger-Tian \cite{CT94} proved the uniqueness of tangent cones at infinity for Ricci flat
manifolds with Euclidean volume growth under certain curvature and integrability assumptions.
This result was later strengthened by Colding-Minicozzi \cite{CM14}, who also proved an analogous result for the mean curvature flow \cite{CM15}---that if a tangent flow at a point is a
cylinder, then it is the unique tangent flow at that point. For the uniqueness of tangent cones in the field of minimal surface, refer to \cite{AA81,S83,W83}. Recently, Colding-Minicozzi \cite{CM21} proved that the cylindrical singularities in the Ricci flow are isolated singularities, and that their appearance implies the uniqueness of the blow-up limit.

In the previous paper of the authors \cite{CMZ21c}, following the methods of \cite{Se06,Ac12,SW15}, it is proved that if a Ricci flow has closed and smooth tangent flow, then the tangent flow must be unique, and is the canonical form of a closed shrinking gradient Ricci soliton. Here the tangent flow is either of an ancient solution at infinity, or of a Ricci flow developing a finite-time singularity at the singular time. Furthermore, we have also proved that the corresponding modified forward or backward Ricci flow converges to the shrinker metric at the rate of $t^{-\beta}$ as $t\to\infty$; see Section 2 for more details.

The modified  Ricci flow is the gradient flow of Perelman's $\mu$-functional. It differs from the Ricci flow by a $1$-parameter family of self-diffeomorphisms. The convergence of the modified Ricci flow obviously implies the Cheeger-Gromov-Hamilton convergence of the scaled sequence of the corresponding Ricci flow, and therefore its convergence in the $\mathbb{F}$-sense (c.f. \cite[Theorem 6.1]{CMZ21a}). Nevertheless, \cite[Theorem 6.1]{CMZ21a} is only a qualitative result, and no comparison between the rates of the two kinds of convergence is made. In this paper, we shall prove that for a Ricci flow with smooth and closed tangent flow, the scaled metric flow has a convergence rate of the form $|\log \lambda|^{-\beta}$, where $\lambda>0$ is the scaling factor, and $\lambda\nearrow\infty$ or $\lambda \searrow 0$ when the tangent flow is a blow-up or blow-down limit, respectively.

Let us recall the settings of \cite{CMZ21c}, which is identical to that of the present article. We consider a metric flow $\displaystyle\left(\XX,\mathfrak{t},(\dist_t)_{t\in I},(\nu_{x\,|\, s})_{x\in \XX,s\in I,s\leq\mathfrak{t}(x)}\right)$ induced by a smooth Ricci flow $(M^n,g_t)_{t\in I}$, where $M$ is a closed manifold. 
See \cite[Definition 5.1]{Bam20b} for the definition of metric flow pairs. 
For any $\lambda>0$, we shall use the notation $\XX^{0,\lambda}$ to represent the scaled metric flow (\cite[Section 6.8]{Bam20b}):
$$\XX^{0,\lambda}=\left(\XX,\lambda^{2}\mathfrak{t},(\lambda\dist_{\lambda^{-2}t})_{t\in \lambda^2 I},(\nu_{x\,|\, \lambda^{-2}s})_{x\in \XX,s\in \lambda^2 I,s\leq\mathfrak{t}(x)}\right),\quad t\in \lambda^2 I.$$
\begin{enumerate}[(1)]
    \item If $(M,g_t)$ is ancient, then we assume that $I=(-\infty,0]$ and hence $\XX=M^n\times(-\infty,0]$. Let us fix a point $(p_0,0)\in M\times\{0\}=\XX_0$ and let $\mu_t:=\nu_{p_0,0\,|\,t}$ be the conjugate heat kernel based at $(p_0,0)$. By Bamler's compactness theorems \cite{Bam20b}, for any sequence $\{\tau_i\}_{i=1}^\infty$ with $\tau_i\nearrow\infty$, we have
    \begin{eqnarray}\label{the_F_convergence}
\left(\XX^{0,(\tau_i)^{-\frac{1}{2}}},\left(\mu_{\tau_it}\right)_{t\in(-\infty,0]}\right)
\xrightarrow{\makebox[1cm]{$\mathbb{F}$}}
\left(\XX^\infty,(\mu^\infty_t)_{t\in -(\infty,0)}\right)
\end{eqnarray}
after passing to a subsequence. Here $\XX^\infty$ is a metric flow over $(-\infty,0]$. If the original ancient solution $(M,g_t)$ has uniformly bounded Nash entropy, that is, if
\begin{eqnarray}\label{Nash_bound}
\text{there exist $(p,t)\in M^n\times(-\infty,0]$ and  $Y<\infty$ such that}\quad \N_{p,t}(\tau)\geq -Y\text{ for all }\tau>0,
\end{eqnarray}
then the singular set of $\XX^\infty$ has parabolic Minkowski codimension no smaller than $4$, and $\left(\XX^\infty,(\mu^\infty_t)_{t\in -(\infty,0)}\right)$ is a metric soliton; see \cite{Bam20c}.
\item If $(M,g_t)$ has finite-time singularity, then we assume that $I=[-T,0)$ and $t=0$ is the singular time. In this case, the metric flow $\XX$ is not defined at $t=0$. We will let $(\mu_t)_{t\in I}$ be a singular conjugate heat kernel based at $t=0$ as defined in \cite{MM15}. By Bamler's compactness theorem \cite{Bam20b} again, for any sequence $\{\tau_i\}_{i=1}^\infty$ with $\tau_i\searrow 0$, we also have that (\ref{the_F_convergence}) holds after passing to a subsequence. In this case, the sequence is automatically noncollapsed, and the limit $\left(\XX^\infty,(\mu^\infty_t)_{t\in -(\infty,0)}\right)$ is a metric soliton whose singular set has parabolic Minkowski codimension no smaller than $4$ \cite{Bam20c}.
\end{enumerate}

The results of this paper are as follows.

\begin{Theorem}\label{Thm_main_1}
Let $(M,g_t)_{t\in(-\infty,0]}$ be an ancient Ricci flow satisfying \eqref{Nash_bound}, where $M$ is a smooth and closed manifold. Suppose that there is a sequence $\tau_i\nearrow\infty$, such that the metric flow $\XX^\infty$ given in \eqref{the_F_convergence} is induced by a smooth and closed Ricci flow. Then the following holds. For any $A>1$, there are constants $\overline\lambda(A)>0$ and $C(A)<\infty$, which also depend on $\XX^\infty$, such that
\begin{equation*}
    \dist_{\mathbb{F}}^{[-A,-1/A]}\left(\left(\XX^{0,\lambda},\mu_{\lambda^{-2}t}\right),(\XX^\infty,\mu^\infty_t)\right)\leq C(-\log\lambda)^{-\beta}\quad \text{whenever}\quad \lambda\leq\overline\lambda,
\end{equation*}
where $\beta>0$ depends only on $\XX^\infty$.
\end{Theorem}

\begin{Theorem}\label{Thm_main_2}
Let $(M,g_t)_{t\in[-T,0)}$ be a Ricci flow, where $M$ is a smooth and closed manifold, and $t=0$ is the singular time. Suppose that there is a sequence $\tau_i\searrow 0$, such that the metric flow $\XX^\infty$ given in \eqref{the_F_convergence} is induced by a smooth and closed Ricci flow. Then the following holds. For any $A>1$, there are constants $\underline\lambda(A)<\infty$ and $C(A)<\infty$, which also depend on $\XX^\infty$, such that
\begin{equation*}
    \dist_{\mathbb{F}}^{[-A,-1/A]}\left(\left(\XX^{0,\lambda},\mu_{\lambda^{-2}t}\right),(\XX^\infty,\mu^\infty_t)\right)\leq C(\log\lambda)^{-\beta}\quad \text{whenever}\quad \lambda\geq\underline\lambda,
\end{equation*}
where $\beta>0$ depends only on $\XX^\infty$.
\end{Theorem}

To summarize the structure of this article, recall that according to the definition of  $\mathbb{F}$-distance (c.f. Definition \ref{Def of F-distance} below), the proofs of Theorem \ref{Thm_main_1} and Theorem \ref{Thm_main_2} consist of two parts: one is to estimate the $1$-Wasserstein distance between corresponding conjugate heat kernels on $\XX^{0,\lambda}$ and on $\XX^\infty$, the other is to estimate the $1$-Wasserstein between the referential conjugate heat flows $\mu_{\lambda^{-2}t}$ and $\mu^\infty_t$; these two estimates are carried out in Section 3 and Section 4, respectively.  Finally, in Section 5, we conclude the proofs of the above theorems.

\section{Background and preparatory results}

In this section, we briefly summarize the main results in \cite{CMZ21c}, establish the current settings, and show some preliminary results for the proof of Theorem \ref{Thm_main_1} and Theorem \ref{Thm_main_2}. For more basic definitions and notations such as the conjugate heat kernel and the Nash entropy, etc., refer to \cite{CMZ21c}.
 
\subsection{Modified Ricci flow}

Under the assumptions of Theorem \ref{Thm_main_1} and Theorem \ref{Thm_main_2}, the tangent flow $\XX^\infty$ is induced by the canonical form of a shrinking gradient Ricci soliton $(M^n,g_o,f_o)$. We normalize the metric $g_o$ and the potential function $f_o$ in the way that
\begin{align*}
    \Ric_{g_o}+\nabla^2f_o&=\tfrac{1}{2}g_o,
    \\\nonumber
    \int(4\pi)^{-\frac{n}{2}}e^{-f_o}dg_o&=1.
\end{align*}
 
We have also followed the idea of \cite{SW15} to consider (backward) \textit{modified Ricci flow} instead. The modified Ricci flow is defined to be the gradient flow of  the functional
\begin{equation*}
    \mu_g:=\inf\left\{\int_M(|\nabla f|^2+R_g+f-n)(4\pi)^{-\frac{n}{2}} e^{-f}dg\,\left|\,\int_M(4\pi)^{-\frac{n}{2}} e^{-f}dg=1 \right\}\right.=\mu(g,1),
\end{equation*}
where $\mu(\cdot,\cdot)$ is Perelman's $\mu$-functional. Specifically, the modified Ricci flow and its backward version are respectively defined as
\begin{align*}
    \partial_t g_t&=\nabla\mu_{g_t}=-2\left(\Ric_{g_t}+\nabla^2 f_{g_t}-\tfrac{1}{2}g_t\right),
    \\
    \partial_t g_t&=-\nabla\mu_{g_t}=2\left(\Ric_{g_t}+\nabla^2 f_{g_t}-\tfrac{1}{2}g_t\right).
\end{align*}
Here (and henceforth) we use $f_g$ to represent the minimizer of $\mu_g$, should it exist and be unique.

In fact, if the minimizer of $\mu_g$ is unique and depends smoothly on $g$, then one may easily convert a Ricci flow into a modified Ricci flow or a backward modified Ricci flow. More precisely, let $(M,g_t)_{t\in I}$, where $M$ is a closed manifold and $I\subset (-\infty,0]$, be a Ricci flow, then $\bg_s$ defined as
\begin{align}\label{modified_RF}
    &\left\{\begin{array}{rl}
    \tg_s &=\ e^{s}g_{-e^{-s}},  \\
     \partial_s\psi_s &=\ -\nabla_{\tg_s}f_{\tg_s}\circ  \psi_s,\\
     \bg_s&=\ \psi_s^*\tg_s,
\end{array}\right.
\end{align}
is a modified Ricci flow, and $\bg_s$ defined as 
\begin{align}\label{backward_modified_RF}
    &\left\{\begin{array}{rl}
    \tg_s &=\ e^{-s}g_{-e^{s}}, \\
     \partial_s\psi_s &=\ \nabla_{\tg_s}f_{\tg_s}\circ  \psi_s,\\
     \bg_s&=\ \psi_s^*\tg_s, 
\end{array}\right.
\end{align}
is a backward modified Ricci flow.

\subsection{\L ojaciewicz argument}

Although the flows in \eqref{modified_RF} and \eqref{backward_modified_RF} are not necessarily well-defined, it is proved in \cite{CMZ21c} that, for the Ricci flow $g_t$ considered in Theorem \ref{Thm_main_1} or Theorem \ref{Thm_main_2}, its corresponding backward modified Ricci flow or modified Ricci flow is well defined, and converges to $g_o$. This is due to the \L ojaciewicz argument of Sun-Wang \cite{SW15} which we shall briefly summarize in this subsection.

\subsubsection{\L ojaciewicz inequality}

Let us recall the \L ojaciewicz inequality in \cite{SW15}. Letting $\R(M)$ be the space of Riemannian metrics on $M$, $k\,\in \mathbb{N}$, $\gamma\in (0,1)$, the $C^{k,\gamma}$ neighborhood of $g_o$ is defined to be
\begin{equation*}
    \mathcal{V}^{k,\gamma}_{\delta}:=\left\{g\in \R(M)\,\Big|\, \|g-g_o\|_{C^{k,\gamma}_{g_o}}\leq \delta\right\}.
\end{equation*}
Sung-Wang \cite{SW15} proved the following \L ojaciewicz inequality:
\begin{Theorem}[Lemma 3.1 in \cite{SW15}]\label{Lojaciewicz}
There is a $C^{k,\gamma}$ ($k\gg 1$) neighborhood $\mathcal U$ of $g_o$, called a \emph{regular neighborhood}, such that for all $g\in\mathcal{U}$, there is a unique $f_g$, and the map $P:f\mapsto f_g$ is analytic on $\mathcal U$. Furthermore, there are constants $C>0$ and $\alpha\in[\frac{1}{2},1)$, such that for any $g\in\mathcal{U}$, we have
$$\|\nabla \mu_g\|_{L^2_g}\geq C|\mu_{g_o}-\mu_g|^\alpha,$$ where $\|\,\cdot\,\|_{L^2_g}$ is the $L^2$ norm taken with respect to $g$. 
\end{Theorem}

\subsubsection{Rate of convergence}

We recall some basic computations related to the \L ojaciewicz arguments, which are very useful in the present article. Assume that $\bg_s$, where $s\in[0,\infty)$, is an \emph{either forward or backward} modified Ricci flow on a closed manifold $M^n$, which converges smoothly to a shrinker metric $g_o$. We also assume that $\bg_s$ is in a small neighborhood of $g_o$ for all $s\ge 0$. By \cite[Corollary 1.3, Corollary 1.4]{CMZ21c}, the Ricci flow in Theorem \ref{Thm_main_1} or Theorem \ref{Thm_main_2} can be converted to such a backward or forward modified Ricci flow through \eqref{backward_modified_RF} or \eqref{modified_RF}, respectively. 

Since $\bg_s$ is always in a neighborhood of $g_o$, it is always regular, and one may estimate the curvature of $\bg_s$ in terms of the curvature of $g_o$. Then, the standard Shi's estimates and the regularity of the metric imply that
$$\left|\nabla^l\Rm_{\bg_s}\right|_{\bg_s}\leq C(l),\quad \left\|\Ric_{\bg_s}+\nabla_{\bg_s}^2f_{\bg_s}-\tfrac{1}{2}\bg_s\right\|_{C^{l,\gamma}_{\bg_s}}\leq C(l),$$
for all $s\in[0,\infty)$ and $l\in\mathbb{N}$; see, for instance, the proof of \cite[Lemma 3.2]{SW15}. Let $\alpha\in[\frac{1}{2},1)$ be the constant in the statement of Theorem \ref{Lojaciewicz} and fix $\beta\in (2-1/\alpha,1)$. Now we consider the forward and backward cases separately.
\\

\noindent\textbf{(1) Backward case.} The argument below is from the appendix of \cite{CMZ21c}, which is modeled after \cite{SW15}. By the standard embedding and interpolation inequality, for any integer $l$, there are integers $p\ge l$ and $N(p)\gg p$, such that
\begin{eqnarray*}
     \left\|
    \partial_s \bg_s
    \right\|_{C^{l,\gamma}_{\bg_s}}
    \le C(l)\left\|
    \partial_s \bg_s
    \right\|_{W^{p,2}_{\bg_s}}
    &\le& C(l,p) \left\|
    \partial_s \bg_s
    \right\|_{L^{2}_{\bg_s}}^{\beta}\left\|\Ric_{\bg_s}+\nabla_{\bg_s}^2f_{\bg_s}-\tfrac{1}{2}\bg_s\right\|^{1-\beta}_{W^{N(p), 2}_{\bg_s}}
    \\
    &\leq &C(l)\left\|
    \partial_s \bg_s
    \right\|_{L^{2}_{\bg_s}}^{\beta},
\end{eqnarray*}
for all $s\geq 0$. Therefore, by the definition of backward modified Ricci flow and Theorem \ref{Lojaciewicz}, we have
\begin{align*}
    \frac{d}{ds}  \left( \mu_{\bg_s} - \mu_{g_o}\right)^{1-(2-\beta)\alpha}
    &= -2 (1-(2-\beta)\alpha) \left( \mu_{\bg_s} - \mu_{g_o}\right)^{-(2-\beta)\alpha} \|\nabla \mu_{\bg_s}\|^2_{L^2_{\bg_s}}\\
    & \le - C(\beta) \left\|
    \partial_s \bg_s
    \right\|_{L^{2}_{\bg_s}}^{\beta}.
\end{align*}
Then
\begin{equation}\label{十十}
   \left\|
    \partial_s \bg_s
    \right\|_{C^{l,\gamma}_{g_o}}\le C(l) \left\|
    \partial_s \bg_s
    \right\|_{C^{l,\gamma}_{\bg_s}}
    \le -C(\beta,l)  \frac{d}{ds}  \left( \mu_{\bg_s} - \mu_{g_o}\right)^{1-(2-\beta)\alpha}.
\end{equation}
Integrating the above inequality from $s$ to $\infty$, we have 
\begin{align}\label{bg covergence rate}
    \|\bg_s-g_o\|_{C^{l,\gamma}_{g_o}}\leq C(\beta,l)\left( \mu_{\bg_s} - \mu_{g_o}\right)^{1-(2-\beta)\alpha}.
\end{align}

On the other hand, we compute using Theorem \ref{Lojaciewicz} again
\begin{eqnarray*}
\frac{d}{ds}\left(\mu_{\bg_s}-\mu_{g_o}\right)^{1-2\alpha}&=&-2(1-2\alpha)\left(\mu_{\bg_s}-\mu_{g_o}\right)^{-2\alpha}\left\|\nabla\mu_{\bg_s}\right\|^2_{L^2_{\bg_s}}
\\
&\geq& 2(2\alpha-1)\left(\mu_{\bg_s}-\mu_{g_o}\right)^{-2\alpha}\cdot C\left(\mu_{\bg_s}-\mu_{g_o}\right)^{2\alpha}
\\
&\geq& 2C(2\alpha-1).
\end{eqnarray*}
Integrating the above inequality, we have
\begin{equation}\label{mu decay}
\mu_{\bg_s}-\mu_{g_o}\leq C(\alpha)s^{-\frac{1}{2\alpha-1}}\quad\text{ for all }\quad s>0.
\end{equation}
Combining \eqref{bg covergence rate} and \eqref{mu decay}, we have 
\begin{align}\label{backwardconvergencerate}
    \|\bg_s-g_o\|_{C^{l,\gamma}_{g_o}}\leq C(\beta,l) s^{-\frac{1-(2-\beta)\alpha}{2\alpha-1}}\quad\text{ for all }\quad s\in(0,\infty).
\end{align}

\noindent\textbf{(2) Forward case.} The forward case is similar to the backward case; more detailed computations can be found in \cite{SW15}. Note that in this case we consider $\mu_{g_o}-\mu_{\bg_s}$ instead, since $\mu_{\bg_s}\leq\mu_{g_o}$ for all $s\in[0,\infty)$. By the same computation as before, we have
\begin{equation*}
   \left\|
    \partial_s \bg_s
    \right\|_{C^{l,\gamma}_{g_o}}\le C(l) \left\|
    \partial_s \bg_s
    \right\|_{C^{l,\gamma}_{\bg_s}}
    \le -C(\beta,l)  \frac{d}{ds}  \left( \mu_{g_o}-\mu_{\bg_s} \right)^{1-(2-\beta)\alpha}
\end{equation*}
and 
\begin{align}\label{forwardconvergencerate}
    \|\bg_s-g_o\|_{C^{l,\gamma}_{g_o}}\leq C(\beta,l) s^{-\frac{1-(2-\beta)\alpha}{2\alpha-1}}\quad\text{ for all }\quad s\in(0,\infty),
\end{align}
where $l$, $\gamma$, and $\beta$ are the same as before.

\subsubsection{The derivative of the \texorpdfstring{$P$}{P} operator}

Recall the map $P$ defined in Theorem \ref{Lojaciewicz} is analytic in a regular neighborhood of a shrinker metric, and hence the minimizer $f_g$ of $\mu_g$ depends on $g$ smoothly, if $g$ is close to a shrinker metric. In the proof of the main theorem of the article, we also need an estimate on the derivative of $P$. In this subsection, by applying the implicit function theorem, we show that $DP$ is a bounded operator in a small neighborhood of a shrinker metric.

Recall that in the proof of  \cite[Lemma 2.2]{SW15}, Sun-Wang consider the operator $L:C^{k,\gamma}(\R(M))\times C^{k,\gamma}(M;\mathbb{R})$ $\longrightarrow C^{k-4,\gamma}(M;\mathbb{R})$ defined by
\begin{equation*}
    L(g,f):=2\Delta^f_g\left(2\Delta_g f+f-|\nabla f|^2+R_g\right)+(4\pi)^{-\frac{n}{2}}\int_{M}e^{-f}dv_g-1,
\end{equation*}
where $k\gg 1$, $\gamma\in(0,1)$, and $C^{k,\gamma}(\R(M))$ stand for all the $C^{k,\gamma}$ Riemannian metrics on $M$ and $\Delta^f_g$ is the weighted laplacian $\Delta_g-\nabla_g f\cdot \nabla_g$. Direct computation shows that the linearization of $L$ in the second component is given by 
\begin{align*}
    DL_{(g,f)}(0,h) = &\ 2\Delta^f_g\left(2\Delta_g^f h+h\right)-(4\pi)^{-\frac{n}{2}}\int_{M} h e^{-f}dv_g
    \\
    &-2\left\langle\nabla_g h, \nabla_g\left(2\Delta_g f+f-|\nabla f|^2+R_g\right)\right\rangle.
\end{align*}
If $(g_o,f_o)$ is a normalized shrinking soliton on $M$, then $L(g_o,f_o)=0$ and the above expression can be simplified to
\begin{equation*}
DL_{(g_o,f_o)}(0,h)=2\Delta^{f_o}_{g_o}\left(2\Delta^{f_o}_{g_o} h+h\right)-(4\pi)^{-\frac{n}{2}}\int_{M} h e^{-f_o}\,dv_{g_o},
\end{equation*}
which is a linear isomorphism from $C^{k,\gamma}(M;\mathbb{R})$ to $C^{k-4,\gamma}(M;\mathbb{R})$ (\cite[Lemma 3.5]{SW15}). Therefore, by the real analytic version of the implicit function theorem (see the proof of \cite[Lemma 2.2]{SW15}), there exists a $C^{k,\gamma}$-neighborhood $\mathcal{U}$ of $g_o$ in $\R(M)$, and a real analytic mapping $P:\mathcal U\to C^{k,\gamma}(M;\mathbb{R})$, such that
\begin{equation*}
    L(g,P(g))=0,\quad\text{ for all }\quad g\in \mathcal{U}.
\end{equation*}
Furthermore,
\begin{equation}\label{def for Q}
    DL_{(g,P(g))}(0,h)=2\Delta^{P(g)}_{g}\left(2\Delta^{P(g)}_{g} h+h\right)-(4\pi)^{-\frac{n}{2}}\int_{M} h e^{-P(g)}\,dv_{g},\quad\text{ for all }\quad g\in \mathcal{U}.
\end{equation}

We shall then estimate the upper bound of 
$DP_g$
for all $g$ close to $g_o$ in the $C^{k,\gamma}$ sense.

\begin{Theorem}\label{derivative bdd of P}
There is a $C^{k,\gamma}$ ($k\gg 1$) neighborhood $\mathcal{V}\subset \mathcal U$ of $g_o$, such that the linearization of $P$
\[
D P_g:\left(C^{k,\gamma}(\mathcal{S}(M)), \|\cdot\|_{C^{k,\gamma}_{g_o}}\right)\longrightarrow \left(C^{k,\gamma}(M;\mathbb{R}), \|\cdot\|_{C^{k,\gamma}_{g_o}}\right)
\]
has uniformly bounded operator norm for any $g\,\in \mathcal{V}$. Here $C^{k,\gamma}(\mathcal{S}(M))$ stands for the space of all $C^{k,\gamma}$ symmetric $2$-tensors on $M$, and $\mathcal U$ is the regular neighborhood of $g_o$ defined in the statement of Theorem \ref{Lojaciewicz}.
\end{Theorem}
\begin{proof}

Let us compute $DL_{(g,f)}(w,0)$. If $g_s$ is a variation of $g$ with $\frac{\partial }{\partial s}g_s\big|_{s=0}=w\in C^{k,\gamma}(\mathcal S(M))$, then, by a straightforward computation, we have
$$\left.\frac{\partial }{\partial s}\right|_{s=0}\Delta_g^f=-\left\langle w,\nabla^2_g\right\rangle_g-\left\langle \left( \div_g w-\tfrac{1}{2}\nabla_g \tr_g w\right), \nabla_g\right\rangle_g+w(\nabla_g f, \nabla_g).$$
Similarly, we have
\begin{eqnarray*}
&&\left.\frac{\partial }{\partial s}\right|_{s=0}\left(2\Delta_g f+f-|\nabla_g f|^2_g+R_g\right)\\
&=&-2\left\langle w,\nabla^2_g f\right\rangle_g-\big\langle \left( 2\div_g w-\nabla_g \tr_g w\right), \nabla_g f\big\rangle_g+w(\nabla_g f, \nabla_g f)\\
&&-\Delta_g \tr_g w+ \div_g \div_g w-\langle w,\Ric_g\rangle_g.
\end{eqnarray*}
Hence if we let $u:=2\Delta_g f+f-|\nabla_g f|^2_g+R_g$, then 
\begin{equation}
\label{first comp linearization}
\begin{split}
    DL_{(g,f)}(w,0)&=-2\left\langle w,\nabla^2_g u\right\rangle_g-\left\langle \left( 2\div_g w-\nabla_g \tr_g w\right), \nabla_g u\right\rangle_g+2w(\nabla_g f, \nabla_g u)\\
    &\,+2\Delta_g^f\Big(-2\left\langle w,\nabla^2_g f\right\rangle_g-\left\langle \left( 2\div_g w-\nabla_g \tr_g w\right), \nabla_g f\right\rangle_g+w(\nabla_g f, \nabla_g f)\\
&\, -\Delta_g \tr_g w+ \div_g \div_g w-\langle w,\Ric_g\rangle_g\Big)+(4\pi)^{-\frac{n}{2}}\int_{M}\tfrac{1}{2}\tr_g w\, e^{-f}\,dv_g.
\end{split}
\end{equation}
Now, by the implicit function theorem we have
\begin{equation}\label{IFT}
    DP_g(w)=-\left(DL_{(g, P(g))}(0,\cdot)\right)^{-1}\circ DL_{(g,P(g))}(w,0).
\end{equation}
Since $P:\mathcal U\to C^{k,\gamma}(M;\mathbb{R})$ is analytic, we can find a $C^{k,\gamma}$ neighborhood $\mathcal V$ of $g_o$ and a constant $C$ depending on $\mathcal V\subset\mathcal U$ but independent of $g\in \mathcal{V}$ such that, by \eqref{first comp linearization}, 
\begin{equation}\label{fallacium}
     \|DL_{(g,P(g))}(w,0)\|_{C^{k-4,\gamma}_{g_o}}\leq C\|w\|_{C^{k,\gamma}_{g_o}}\quad\text{ for all }\quad w\in C^{k,\gamma}(\mathcal S(M)).
\end{equation}
Note that here we have used the fact that if $f=P(g)$, then $\|f\|_{C^{k,\gamma}_{g_o}}$ is bounded independent of $g\in\mathcal V$.\\

\noindent\textbf{Claim:}
There exists a $C^{k,\gamma}$ neighborhood of $\mathcal V'\subseteq \mathcal V$ of $g_o$ such that $DL_{(g,P(g))}(0,\cdot)$ is also a linear isomorphism from $C^{k,\gamma}(M;\mathbb{R})$ to $C^{k-4,\gamma}(M;\mathbb{R})$, and
\begin{equation*}
  \left \|\left(DL_{(g, P(g))}(0,\cdot)\right)^{-1} \right\|\leq  2\left\|\left(DL_{(g_o, f_o)}(0,\cdot)\right)^{-1} \right\|\quad \text{  for all   }\quad g\in \mathcal{V'}.
\end{equation*}
\begin{proof}[Proof of the claim]
Let $\mathcal B_1$ be the set of all bounded linear operators from $C^{k,\gamma}(M;\mathbb{R})$ to $C^{k-4,\gamma}(M;\mathbb{R})$ endowed with the operator norm, i.e. 
\[
\|A\|_{\mathcal B_1}=\sup_{u\neq 0}\frac{\|A(u)\|_{C^{k-4,\gamma}_{g_o}}}{\|u\|_{C^{k,\gamma}_{g_o}}}.
\]

With the induced topology by this norm, $\mathcal B_1$ is a Banach space. Similarly, we denote by $\mathcal B_2$ and $\mathcal B_3$, the Banach spaces of all bounded linear operators from $C^{k-4,\gamma}(M;\mathbb{R})$ to $C^{k,\gamma}(M;\mathbb{R})$, and of all bounded linear operators from $C^{k,\gamma}(M;\mathbb{R})$ to $C^{k,\gamma}(M;\mathbb{R})$, respectively, with the corresponding operator norms. We consider the map $Q:\mathcal V\longrightarrow \mathcal B_1$ given by 
\[
Q(g):= DL_{(g,P(g))}(0,\cdot).
\]
Since, $P:\mathcal U\to C^{k,\gamma}(M)$ is continuous (c.f. \cite[Theorem 10.1]{J05}), it can be seen from \eqref{def for Q} that $Q$ is also continuous. For any $g\in \mathcal V$, we have
\begin{equation}\label{Q eqn}
    \begin{split}
        Q(g)&= Q(g_o)+Q(g)-Q(g_o)\\
        &=Q(g_o)\circ\left[I+Q(g_o)^{-1}\circ\left(Q(g)-Q(g_o)\right)\right],
    \end{split}
\end{equation}
where $I$ is the identity map
on $C^{k,\gamma}(M;\mathbb{R})$.
As $Q(g_o)$ is a linear isomorphism, we have that $(Q(g))^{-1} \in\mathcal B_2$ if and only if $I+Q(g_o)^{-1}\circ\left(Q(g)-Q(g_o)\right)$ is an invertible bounded linear map from $C^{k,\gamma}(M;\mathbb{R})$ to $C^{k,\gamma}(M;\mathbb{R})$. It suffices to check that for $g$ sufficiently close to $g_o$, $\|E(g)\|_{\mathcal B_3}<1$, where $E(g):=Q(g_o)^{-1}\circ\left(Q(g)-Q(g_o)\right)$, since in this case \cite[Lemma 7.22]{J05}, we know that $I+E(g)$ is invertible and
\begin{equation}\label{norm of E}
\left(I+E(g)\right)^{-1}=\sum_{k=0}^{\infty} (-1)^k\left(E(g)\right)^{k} \text{  and  }  \left\|\left(I+E(g)\right)^{-1}\right\|_{\mathcal B_3}\le \frac{1}{1-\left\|E(g)\right\|_{\mathcal B_3}}.
\end{equation}
By the continuity of $Q$, there is a $C^{k,\gamma}$ neighborhood of $\mathcal V'\subseteq \mathcal V$ of $g_o$ such that
\[
\left\|E(g)\right\|_{\mathcal B_3}\le \left\|Q(g_o)^{-1}\right\|_{\mathcal B_2}\left\|Q(g)-Q(g_o)\right\|_{\mathcal B_1}<\tfrac{1}{2}\quad\text{ for all }\quad g\in\mathcal{V}'.
\]
Hence on $\mathcal V'$, $\left(I+E(g)\right)$ as well as $Q(g)$, are invertible. Moreover by \eqref{Q eqn} and \eqref{norm of E}, we have
\begin{equation*}
    \begin{split}
      \left \|\left(DL_{(g, P(g))}(0,\cdot)\right)^{-1} \right\|_{\mathcal B_2}&=  \left\|Q(g)^{-1}\right\|_{\mathcal B_2}
      \le \left\|\left(I+E(g)\right)^{-1}\right\|_{\mathcal B_3}\left\|Q(g_o)^{-1}\right\|_{\mathcal B_2}\\
      &\le 2\left\|Q(g_o)^{-1}\right\|_{\mathcal B_2}= 2\left\|\left(DL_{(g_o, f_o)}(0,\cdot)\right)^{-1} \right\|_{\mathcal B_2}.
    \end{split}
\end{equation*}
This completes the proof of the claim.
\end{proof}

Hence, by replacing $\mathcal V$ by $\mathcal V'$ if necessary, 
there exists a $C$ independent of $g\in\mathcal V$, such that for all $g\in\mathcal V$, we have
\begin{equation}\label{fallacitas}
    \left \|\left(DL_{(g, P(g))}(0,\cdot)\right)^{-1}(h) \right\|_{C^{k,\gamma}_{g_o}}\leq C\left\|h\right\|_{C^{k-4,\gamma}_{g_o}}\quad\text{ for all }\quad h\in\mathcal C^{k-4,\gamma}(M).
\end{equation}
Finally, by \eqref{IFT}, (\ref{fallacium}), and (\ref{fallacitas}), the theorem follows.

\end{proof}

\subsection{Definition of the
\texorpdfstring{$\mathbb{F}$}{F}-distance}

The $\mathbb{F}$-distance is a generalization of the Gromov-$W_1$-Wasserstein distance studied by Sturm \cite{St06a,St06b}.  To help the reader grasp the main ingredients of the proof, we include the definition of Bamler's $\mathbb{F}$-distance in this subsection. Note that Definition \ref{Def of F-distance} is slightly simpler than the original definition in \cite[Section 5]{Bam20b}, but it is more appropriate for application and is sufficient for our purpose.


The definition of $\mathbb{F}$-convergence is involved in the notions of coupling and $1$-Wasserstein distance between probability measures. Let $X$ and $Y$ be metric spaces. We denote by $\PP(X)$ the space of probability measures on $X.$
For any $\mu\in \PP(X)$ and $\nu\in \PP(Y),$ we denote by $\Pi(\mu,\nu)$ the space of \emph{couplings} between $\mu$ and $\nu$, namely, the set of all the probability measures $q\in \PP(X\times Y)$ satisfying
\[
    q(A\times Y)=\mu(A),\quad
    q(X\times B)=\nu(B),
\]
for any measurable subsets $A\subset X$ and $ B\subset Y$. The $1$-Wasserstein distance between $\mu,\nu\in \PP(X)$ is defined to be
\[
    \dist_{W_1}(\mu,\nu)
    := \inf_{q\in \Pi(\mu,\nu)} \int_{X\times X} \dist(x,y) \, dq(x,y).
\]
By the Kantorovich-Rubinstein Theorem, this definition is equivalent to
\[
    \dist_{W_1}(\mu,\nu)
    = \sup_{f}\left( \int f\, d\mu 
    - \int f\, d\nu\right),
\]
where the supremum is taken over all bounded $1$-Lipschitz functions $f$ on $X.$

Let $I$ be an interval and $(\XX^i,(\mu^i_t)),i=1,2,$ be metric flow pairs defined over $I$. 

\begin{Definition}\label{Def of F-distance}
Let $J\subset I.$
The $\mathbb{F}$-distance between $(\XX^1,(\mu^1_t))$ and $(\XX^2,(\mu^2_t))$ (uniform over $J$), denoted by
\[
    \dist_{\mathbb{F}}^J\left((\XX^1,(\mu^1_t)),
    (\XX^2,(\mu^2_t))\right),
\]
is defined to be the infimum of number $r>0$ with the following properties:
There are a measurable set $E\subset I$, couplings $q_t\in \Pi(\mu_t^1,\mu_t^2),$ and embeddings $\phi^i_t:\XX^i_t\to \left(Z_t, \dist^{Z_t}\right)$ for $i=1,2, t\in I\setminus E,$ where $\left(Z_t, \dist^{Z_t}\right)$ is a complete separable metric space,  such that $J\subset I\setminus E$, and
\begin{enumerate}
    \item $|E|\le r^2$;
    \item For any $s,t\in I\setminus E, s<t$, 
    \[
        \int_{\XX_t^1\times \XX_t^2} \dist^{Z_s}_{W_1}\left(\phi^1_{*s}\nu^1_{x_1|s},
        \phi^2_{*s}\nu^2_{x_2|s}\right)\, dq_t(x_1,x_2)
        \le  r.
    \]
\end{enumerate}
The tuple $\left(Z_t,\dist^{Z_t},\{\phi^i_t\}_{i=1,2}\right)_{t\in I\setminus E}$ is called a \emph{correspondence}.
\end{Definition}

\section{Almost monotonicity of the \texorpdfstring{$1$}{W1}-Wasserstein distance}

As indicated by Definition \ref{Def of F-distance}, an important ingredient in the estimate of the $\mathbb{F}$-distance is to compare the corresponding conjugate heat kernels in the metric flows. In this section, we consider the $1$-Wasserstein distance between two conjugate heat kernels on two different Ricci flows sufficiently close to each other. It turns out that an almost monotonicity property similar to \cite[Lemma 2.7]{Bam20a} is available. This almost monotonicity further implies that two Ricci flows close to each other in the $C^2$-sense are also close in the $\mathbb{F}$-sense, given that the referntial conjugate heat flows are close to each other.

Throughout this section, we shall let $M^n$ be a  closed $n$-dimensional manifold and $(M^n,g_{i,t})_{t\in I}$, $i=1,2$, be Ricci flows defined on the same time interval $I$. Let $\bg$ be a background metric on $M$. We make the assumption that for all $t\in I$ and $i=1,2$, it holds that
\begin{gather}\label{iniquitas}
C_0^{-1} \bg \le g_{i,t} \le C_0 \bg,
    \quad
    \operatorname{diam}(\bg) \le D,
    \quad
    \sup_{M\times I}|{\Rm}_{g_i}| \le \Lambda,
\end{gather}
for some positive constants $C_0$, $D$, and $\Lambda$. The background metric $\bg$ can be taken as, for instance, $g_{1,t_0}$ for some $t_0\in I$. Unless otherwise specified, the tensor norms in this section are always computed using  the metric $\bg$, and we shall suppress the metric specification in the norm notations, i.e., $\|\cdot\|_{C^2}=\|\cdot\|_{C^2_{\bg}}$, etc. 

Let us assume that 
\begin{eqnarray}\label{iniquitas_2}
 \sup_{t\in I} \left\|
        g_{1,t}-g_{2,t}
    \right\|_{C^2} \le \varepsilon \ll 1.
\end{eqnarray}
Then, following \cite{Bam20b}, we may easily construct a correspondence $\left(Z_t,\dist^{Z_t},\{\phi^i_t\}_{i=1,2}\right)_{t\in I}$ as follows. Let $Z_t=M_1\sqcup M_2$ for $t\in I$, where $M_1=M_2=M$ and, for notational convenience, we have added the subindex to indicate that $M_i$ corresponds to the underlying manifold of the Ricci flow $g_{i,t}$. For all $z_1\in M_1,z_2\in M_2,$ we define
\[
    \dist^{Z_t}(z_1,z_2)
    :=\dist^{Z_t}(z_2,z_1)
    := \inf_{w\in M}\left( \dist_{g_{1,t}}(z_1,w)
    + \dist_{g_{2,t}}(w,z_2)\right) + \varepsilon.
\]
Then it is straightforward to check that $\left(Z_t,\dist^{Z_t}\right)$ is a compact metric space.
Let $\phi^i_t:\left(M, \dist_{g_{i,t}}\right)\to \left(Z_t,\dist^{Z_t}\right)$ be the canonical embedding. Clearly, for any $x\in M$, we have
\begin{equation}\label{anothernonsenseotherthanthepreviousnonsense}
     \dist^{Z_t}(\phi^1_t(x),\phi^2_t(x))
    = \varepsilon.
\end{equation}
By a slight abuse of notations, we still write $M_i=\phi^i_t(M)\subset Z_t$, $i=1,2$.

\begin{Proposition}\label{quantitas W-dist}
Let $(M^n,g_{i,t})_{t\in I}$ be Ricci flows satisfying (\ref{iniquitas}) and (\ref{iniquitas_2}), where $M$ is a closed manifold. For all $x_1,x_2\in M$ and for all $s\le s'\le s_1,s_2,$ where $ s,s',s_1,s_2\in I,$ if
\[
    \mathcal{N}^2_{x_2,s_2}(s_2-s) \ge -Y\quad
    \text{ or }\quad
    \mathcal{N}^1_{x_1,s_1}(s_1-s) \ge -Y,
\]
then
\begin{align}\label{ineq: W1 almost mono}
     & \dist^{Z_{s}}_{{\rm W}_1}\left(
    \phi^1_{s*} \nu^1_{x_1,s_1\,|\,s},
    \phi^2_{s*} \nu^2_{x_2,s_2|s}
    \right)
    \\\nonumber
    \le\ &
    (1+C\varepsilon)\dist^{Z_{s'}}_{{\rm W}_1}\left(
    \phi^1_{s' *} \nu^1_{x_1,s_1\,|\,s'},
    \phi^2_{s' *} \nu^2_{x_2,s_2\,|\,s'}
    \right) + C \varepsilon (1+{s'-s}),
\end{align}
   where $\N^i$ and $\nu^i$ stand for the Nash entropy and the conjugate heat kernel on the Ricci flow $(M,g_{i,t})_{t\in I}$, respectively, and $C$ is a constant depending only on $C_0$, $D$, $Y$, $\Lambda$, and $n$.
In particular, setting $s_1=s_2=s'=t$, we have
\begin{eqnarray}
\dist^{Z_{s}}_{{\rm W}_1}\left(
    \phi^1_{s*} \nu^1_{x_1,t\,|\,s},
    \phi^2_{s*} \nu^2_{x_2,t\,|\,s}
    \right)
\le  (1+C\varepsilon)\dist^{Z_t}(\phi^1_t(x_1),\phi^2_t(x_2))
+C \varepsilon(1+{t-s}).
\end{eqnarray}
\end{Proposition}

\begin{proof}
The idea of the proof is similar to that of \cite[Lemma 2.7]{Bam20a}. Fix a point $o\in M$ and write
\[
    d\nu^i_t:=d\nu^i_{x_i,s_i\,|\,t}
    = v_{i,t}\, dg_{i,t},
\]
where $d\nu^{i}_{x_i,s_i\,|\,t}$ is the conjugate heat kernel on the Ricci flow $(M,g_{i,t})$ based at the point $(x_i,s_i)$. By the symmetry in the statement of this proposition, we may assume that
\[
    \mathcal{N}^2_{x_2,s_2}(s_2-s)
    \ge -Y.
\]
Let $F:Z_s\to \mathbb{R}$ be an arbitrary $1$-Lipschitz function. By replacing $F$ by $F-F(\phi^1_{s}(o))$ if necessary, we may assume that $F\circ\phi^1_{s}(o)=0$. Then we write $u^0_{i}=F\circ\phi^i_s$ and solve the heat equations
\begin{gather*}
    \Box_{g_{i,t}} u_{i,t} = 0\quad
    \text{ on } \quad M\times [s,s'],\\
    u_{i,s}=u^0_i,
\end{gather*}
for $i=1$, $2$. In fact, we will not need $u_{2,t}$ in the following. Obviously, $u^0_i$, $i=1,2$, are both $1$-Lipschitz, and hence
 \cite[Lemma 2.5]{Bam20a} implies that $|\nabla u_{i,t}|_{g_{i,t}}\le 1$ for all $t\in [s,s']$ and $i=1,2$. 
By our assumption on $F$, we have  $u^0_{1}(o)=0$. Then the parabolic maximum principle and the bound of diameter in \eqref{iniquitas} imply
\[
    |u^0_{1}|\le {\rm  diam}(g_{1,t})
    \le C_0 D
\]
and
\begin{eqnarray}\label{finis desuper}
 \sup_{M}|u_{1,t}|\le \sup_{M} |u^0_1| \leq C_0D\quad\text{ for }\quad t\in[s,s']\quad\text{ and }\quad i=1,2.
\end{eqnarray}
Letting $\bar s = (s+s')/2$, we clearly have
\[
    \left.\int u_{1,t} d\nu^2_{t}\,\right|_{t=s}^{t=s'}
    =\left.\int u_{1,t} d\nu^2_{t}\,\right|_{t=\bar s}^{t=s'}
    +\left.\int u_{1,t} d\nu^2_{t}\,\right|^{t=\bar s}_{t=s}.
\]
We shall estimate the two terms above separately. First of all, we have
\begin{align*}
    \left.\int u_{1,t} d\nu^2_{t}\,\right|_{t=\bar s}^{t=s'}
    &= \int_{\bar s}^{s'} \int_M 
    \left(\Box_{g_{2,t}} u_{1,t}\right) d\nu^2_{t}dt.
\end{align*}
Note that for $t\in (s,s']$, we have
\begin{align*}
   \left|\Box_{g_{2,t}} u_{1,t}\right|
    &=\left|\Box_{g_{2,t}} u_{1,t}
    - \Box_{g_{1,t}} u_{1,t}\right|
    = \left|\Delta_{g_{2,t}} u_{1,t}
    - \Delta_{g_{1,t}} u_{1,t}\right|\\
    &\le C \left\|g_{1,t}-g_{2,t}\right\|_{C^2} \cdot \sup_M\left(
    |\nabla^2 u_{1,t}|_{g_{1,t}}
    +
    |\nabla u_{1,t}|_{g_{1,t}}
    \right) \\
    & \le C \varepsilon 
   \frac{1+\sqrt{t-s}}{t-s},
\end{align*}
where $C$ depends on $C_0$, $D$, $\Lambda$, $n$, and we have applied the standard Bando-Bernstein-Shi estimates for the heat equation (see, e.g., \cite[Lemma 9.14]{Bam20b}). It follows that
\begin{align}
\label{ineq: second half}
    \left|\left.\int u_{1,t} d\nu^2_{t}\right|_{t=\bar s}^{t=s'}\,\right|
    &\le  C \varepsilon (1+ \sqrt{s'-s})\leq C\varepsilon(1+s'-s).
\end{align}

On the other hand, for any $t\in [s,s']$, we have
\begin{align*}
  \left|\int u_{1,t} d\nu^2_{t}
    - \int u_{1,t} v_{2,t}\, dg_{1,t}\right|
&\le C_0D \int v_{2,t}\,\left|dg_{2,t}-dg_{1,t}\right|  \\
&\le  C(C_0,D) \varepsilon \int v_{2,t}\,dg_{2,t} 
\le C \varepsilon,
\end{align*}
where we have applied (\ref{finis desuper}). Hence we have
\begin{align}\label{secondhalf_piece1}
    \left|\left.\int u_{1,t} d\nu^2_{t}\,\right|^{t=\bar s}_{t=s}
    - \left.\int u_{1,t} v_{2,t}\, dg_{1,t}\,\right|^{t=\bar s}_{t=s}\right|
    &\le  C \varepsilon
\end{align}
 Next,
\begin{align}\label{secondhalf_piece2}
   \left.\int u_{1,t} v_{2,t}\, dg_{1,t}\,\right|^{t=\bar s}_{t=s}
   &= -\int_{s}^{\bar s} u_{1,t}\left(\Box^*_{g_{1,t}}v_{2,t}\right)\, dg_{1,t},
\end{align}
and we also have
\begin{align}\label{secondhalf_piece3}
  \left|\Box^*_{g_{1,t}}v_{2,t}\right|
  & = \left|\Box^*_{g_{1,t}}v_{2,t}
   - \Box^*_{g_{2,t}}v_{2,t}\right|\\\nonumber
    &\le  \left|\Delta_{g_{1,t}}v_{2,t}
   - \Delta_{g_{2,t}}v_{2,t}\right|
   + |R_{g_{1,t}}-R_{g_{2,t}}|v_{2,t}\\\nonumber
   &\le C \varepsilon 
   \left(
    |\nabla^2 v_{2,t}|_{g_{2,t}}
    +|\nabla v_{2,t}|_{g_{2,t}}
    + v_{2,t}
   \right).
\end{align}
Writing $v_{2,t}=(4\pi (s_2-t))^{-n/2} e^{-f_t}$, we have
\[
    \nabla v_{2,t} = - v_{2,t}\nabla f_t,\quad
    \nabla^2 v_{2,t}
    = v_{2,t}\left(\nabla f_t \otimes \nabla f_t - \nabla^2 f_t\right).
\] 
Therefore, applying \cite[Proposition 5.2]{Bam20c} and \eqref{iniquitas}, we have
\begin{align}\label{secondhalf_piece4}
    &\int_{s}^{\bar s}\int_M \left(|\nabla^2 v_{2,t}|_{g_{2,t}}
+ |\nabla v_{2,t}|_{g_{2,t}}
+ v_{2,t}\right)\, dg_{1,t}dt
\\\nonumber
\leq\ & C\int_{s}^{\bar s}\int_M\left(|\nabla^2 f_t|+|\nabla f_t|^2+|\nabla f_t|+1\right)v_{2,t}\, (1+\varepsilon)dg_{2,t}dt
\\\nonumber
\leq\ & C\left(\int_s^{\bar s}\int_M(s_2-t)|\nabla^2 f_t|^2d\nu^2_tdt\right)^{\frac{1}{2}}\left(\int_s^{\bar s}\int_M\frac{1}{s_2-t}d\nu^2_tdt\right)^{\frac{1}{2}}
\\\nonumber
&+\,C\left(\int_s^{\bar s}\int_M(s_2-t)|\nabla f_t|^4d\nu^2_tdt\right)^{\frac{1}{2}}\left(\int_s^{\bar s}\int_M\frac{1}{s_2-t}d\nu^2_tdt\right)^{\frac{1}{2}}
\\\nonumber
&+\,C\left(\int_s^{\bar s}\int_M|\nabla f_t|^2d\nu^2_tdt\right)^{\frac{1}{2}}\left(\int_s^{\bar s}\int_M1\,d\nu^2_tdt\right)^{\frac{1}{2}}+C(s'-s)
\\\nonumber
\le\ & C(Y)(1+ \sqrt{s'-s}+s'-s)\leq C(Y)(1+s'-s).
\end{align}
Combining (\ref{finis desuper}), (\ref{secondhalf_piece1}), (\ref{secondhalf_piece2}), (\ref{secondhalf_piece3}), and (\ref{secondhalf_piece4}), we have
\begin{equation}
\label{ineq: first half}
   \left|\left.\int u_{1,t} d\nu^2_{t}\,\right|^{t=\bar s}_{t=s}
     \right|
     \le C(C_0,D,Y)\varepsilon (1+ {s'-s}). 
\end{equation}
It then follows from \eqref{ineq: second half} and \eqref{ineq: first half} that
\begin{eqnarray}\label{summa differentiae}
 \left|\left.\int u_{1,t} d\nu^2_{t}\right|^{t=s'}_{t=s}
     \right|
     \le C \varepsilon
     \left(1 + {s'-s}\right).
\end{eqnarray}

By (\ref{iniquitas}) again, we have
\[
   \left| |\nabla u_{1,s'}|^2_{g_{2,s'}} - |\nabla u_{1,s'}|^2_{g_{1,s'}}
   \right|
   \le C \varepsilon,
\]
and hence $u_{1,s'}$ is a $(1+C\varepsilon)$-Lipschitz function on $\left(M,\dist_{g_{2,s'}}\right)$. Therefore, we may find  a $(1+C\varepsilon)$-Lipschitz function $U$ on $Z_{s'}$, such that $U\circ\phi^i_{s'}=u_{1,s'}$ for both $i=1,2.$
It follows that
\[
    \int u_{1,s'} \, d\nu_{s'}^1
    - \int u_{1,s'} \, d\nu_{s'}^2
    \le (1+C\varepsilon)
    \dist^{Z_{s'}}_{{\rm W}_1}\left(
    \phi^1_{s'*} \nu^1_{s'},
    \phi^2_{s'*} \nu^2_{s'}
    \right).
\]
Recall
\[
    \int u_{1,s'} \, d\nu_{s'}^1
    = \int u_{1}^0 \, d\nu_{s}^1,
\]
thus, by (\ref{summa differentiae}), we have
\begin{align*}
    & \int u_{1}^0\,d\nu^1_{s}
    - \int u_{1}^0\, d\nu^2_{s}
    \\
    =\ &\int u_{1,s'}\,d\nu^1_{s'}
    - \int u_{1,s'}\, d\nu^2_{s'} +\left.\int u_{1,t} d\nu^2_{t}\,\right|^{t=s'}_{t=s}
    \\
    \le\ & (1+C\varepsilon)
    \dist^{Z_{s'}}_{{\rm W}_1}\left(
    \phi^1_{s'*} \nu^1_{s'},
    \phi^2_{s'*} \nu^2_{s'}
    \right)
    +C \varepsilon
     \left(1 + {s'-s}\right).
\end{align*}
Since $F$ is $1$-Lipschitz on $Z_{s}$ and because of \eqref{anothernonsenseotherthanthepreviousnonsense}, for any $x\in M$ we have
\[
    \left|u^0_1(x)-u^0_2(x)\right|
    = \left|F(\phi^1_{s}(x))-F(\phi^2_{s}(x))\right|
    \le \dist^{Z_{s}}(\phi^1_{s}(x),\phi^2_{s}(x)) = \varepsilon.
\]
It follows that
\begin{align*}
    &\int_{Z_{s}} F \, d(\phi^1_{s*} \nu^1_{s})
    - \int_{Z_{s}} F\,
    d(\phi^2_{s*} \nu^2_{s})
    \\
    =\ &  \int_{M} u^0_1 \, d \nu^1_{s}
    - \int_{M} u_2^0
    d \nu^2_{s}
    \le
    \int_{M} u^0_1 \, d \nu^1_{s}
    - \int_{M} u_1^0
    d \nu^2_{s} + \varepsilon\\
    \le\ & 
    (1+C\varepsilon)
    \dist^{Z_{s'}}_{{\rm W}_1}\left(
    \phi^1_{s'*} \nu^1_{s'},
    \phi^2_{s'*} \nu^2_{s'}
    \right)
    +C \varepsilon
     \left(1 + {s'-s}\right).
\end{align*}
Since $F$ is arbitrary, the almost monotonicity formula \eqref{ineq: W1 almost mono} is proved. 

\end{proof}

Next, applying Proposition \ref{quantitas W-dist}, we show the following corollary, which is analogous to \cite[Lemma 5.19]{Bam20b}.

\begin{Corollary}
\label{prop: F dist}
Suppose that the same assumptions of the previous proposition hold.
Let $(\mu^i_t)_{t\in I}$ be a conjugate heat flow on $(M,g_{i,t})_{t\in I}$, where $i=1,2$, and let $J\subset I.$
Then 
$$\dist_{\mathbb{F}}^J\big(((M,g_{1,t})_{t\in I},(\mu^1_t)_{t\in I}),
((M,g_{2,t})_{t\in I},(\mu^2_t)_{t\in I})\big)
\le (1+C\varepsilon)r + C \varepsilon (1+|I|),$$
where $r>0$ is any number with the property that there is a measurable subset $E\subset I$ such that $J\subset I\setminus E$ and
\[
    |E|< r^2,\quad
    \sup_{t\in I\setminus E} 
    \dist^{Z_t}_{{\rm W}_1}(\phi^1_{t*}\mu_t^1,\phi^2_{t*}\mu_t^2)
    < r.
\]
\end{Corollary}

\begin{proof}
Let $r>0$ be a number satisfying the property stated in the proposition.
For any $t\in I\setminus E,$ let $q_t\in \Pi(\mu_t^1,\mu_t^2)$ be a coupling satisfying
\[
    \int_{M\times M} \dist^{Z_t}(\phi^1_t(x), \phi^2_t(y))
    \, dq_t(x,y)
    < r.
\]
Then for any $s,t\in I\setminus E$, $s\le t$, by applying Proposition \ref{quantitas W-dist}, we have
\begin{align*}
    & \int_{M\times M}
    \dist^{Z_{s}}_{{\rm W}_1}\left(
    \phi^1_{s*} \nu^1_{x,t|s},
    \phi^2_{s*} \nu^2_{y,t|s}
    \right)\, dq_t(x,y)\\
\le &\  (1+C\varepsilon)\int_{M\times M}\dist^{Z_t}(\phi^1_t(x),\phi^2_t(y))\,dq_t(x,y)
+C \varepsilon(1+t-s)\\
< &\ (1+C\varepsilon) r + C \varepsilon(1+|I|).
\end{align*}
The conclusion follows from Definition \ref{Def of F-distance}.
\end{proof}

\section{Convergence rate of the referential conjugate heat flow}

It is obvious from Corollary \ref{prop: F dist} that, to estimate the $\mathbb{F}$-convergence rate for the Ricci flows in Theorem \ref{Thm_main_1} and Theorem \ref{Thm_main_2}, we need also to estimate the convergence rate of the referential conjugate heat flow $\mu_t$; this is the goal of the present section.

Let $(M^n,g_t)_{t\in I}$ be the Ricci flow in the statement either of Theorem \ref{Thm_main_1}, in which case $I=(-\infty,0]$, or, of Theorem \ref{Thm_main_2}, in which case $I=(-T,0]$. Let 
\begin{equation}\label{referential CHF}
    d\mu_t:=u(\cdot,t)dg_t:=(4\pi|t|)^{-\frac{n}{2}}e^{-f_t}dg_t,\quad t\in I
\end{equation}
be the referential conjugate heat flow in the statement of Theorem \ref{Thm_main_1} or Theorem \ref{Thm_main_2}.  Note that in the case of Theorem \ref{Thm_main_1}, $\mu_t:=\nu_{p_0,0\,|\,t}$ for some arbitrarily fixed $p_0\in M$; in the case of Theorem \ref{Thm_main_2}, $u$ is a singular conjugate heat kernel constructed in \cite{MM15}. Let $(M,g_o,f_o)$ be the normalized shrinker that generates the unique tangent flow in Theorem \ref{Thm_main_1} or Theorem \ref{Thm_main_2} (c.f. \cite[Corollary 1.3, Corollary 1.4]{CMZ21c}).

Let $\bg_s$, where $s\in[0,\infty)$, be the (backward) modified Ricci flow constructed from $g_t$, as demonstrated in Section 2.1. Furthermore, $\bg_s$ satisfies all the properties listed in Section 2.2.2. For the sake of simplicity, we write 
\begin{equation}\label{bg vs go}
\|\bg_s-g_o\|_{C^{k,\gamma}_{g_o}}\leq Cs^{-\theta} \quad\text{ for all  } \quad s>0,
\end{equation}
where $\theta:=\frac{1-(2-\beta)\alpha}{2\alpha-1}$, $k\gg 1$, and $\gamma\in (0,1)$; in particular, we choose $k$ and $\gamma$ to be the constants in Theorem \ref{Lojaciewicz} and Theorem \ref{derivative bdd of P}.

To simplify our argument, we may, by shifting $s$ if necessary, without loss of generality, assume that $\bg_s$ is very close to $g_o$, so that it satisfies Theorem \ref{Lojaciewicz} and Theorem \ref{derivative bdd of P} for all $s\geq 0$. Furthermore, if $\bg_s$ is very close to $g_o$ for all $s\geq 0$, we also have that $f_{\bg_s}$ is very close to $f_o$ for all $s\geq 0$. Therefore, in view of the fact that $\lambda_1\left(-\Delta_{g_o}^{f_o}\right)> \frac{1}{2}$ (c.f. \cite[Lemma 3.5]{SW15}), where $$\Delta_g^f:=\Delta_g-\langle\nabla_g f,\nabla_g\cdot\,\rangle$$
is the drifted laplacian operator, and $\lambda_1$ is the first nonzero eigenvalue, we also have that 
$$\lambda_1\left(-\Delta_{\bg_s}^{f_{\bg_s}}\right)>\frac{1}{2}\quad\text{ for all }\quad s\geq 0.$$
As a consequence, we have the following Neumann type Poincar\'e inequality.

\begin{Lemma}
For any $s\geq 0$, if $v\in C^\infty(M)$ satisfies $\displaystyle \int_Mve^{-f_{\bg_s}}d\bg_s=0$, then
\begin{equation}\label{Poincare}
    \int_M v^2 e^{-f_{\bg_s}}d\bg_s\leq 2 \int_M |\nabla_{\bg_s}v|^2 e^{-f_{\bg_s}}d\bg_s. 
\end{equation}
\end{Lemma}

Next, we also modify the referential conjugate heat flow. Let us define
\begin{align}
    \tf_s&:=\left\{\begin{array}{ll}
    f_{-e^s} & \text{in the backward case}
    \\
    f_{-e^{-s}} &\text{in the forward case}
    \end{array}\right.\\\nonumber
    \bf_s&:=\tf_s\circ\psi_s
\end{align}
where $f_t$ is the function in \eqref{referential CHF}, and $\psi_s$ is the $1$-parameter family of self-diffeomorphism defined by (\ref{backward_modified_RF}) in the backward case, or by (\ref{modified_RF}) in the forward case. Perelman's monotonicity formula obviously implies that
\begin{equation*}
    \bf_s\longrightarrow f_o\quad \text{ smoothly, }\quad \text{when } s\to \infty.
\end{equation*}
Comsequently, we also have
\begin{equation}\label{coarseconvergenceoff2}
    \left|\bf_s-f_{\bg_s}\right|\to 0\quad \text{ uniformly.}
\end{equation}
Our goal is to estimate the rate of these convergences. Now we split our argument into two cases.
\\

\noindent\textbf{(1) Backward case.} It follows from the definition of $u$ that 
\begin{equation}\label{eqn for tf}
\partial_s \tilde{f}_s-\Delta_{\tg_s}\tilde{f}_s+|\nabla_{\tg_s}\tilde{f}_s|^2-R_{\tg_s}+\frac{n}{2}=0.
\end{equation}
We will compare $\bf_s$ to $f_{\bg_s}$. To this end, we start with deriving an evolution equation satisfied by $\bf_s-f_{\bg_s}$. By virtue of \eqref{eqn for tf}, we have
\begin{eqnarray*}
\partial_s \bf_s&=&\left(\partial_s \tf_s\right)\circ \psi_s+\langle\nabla_{\tg_s}\tf_s,\nabla_{\tg_s}f_{\tg_s}\rangle\circ\psi_s\\
&=&\left(\partial_s \tf_s\right)\circ \psi_s+\langle\nabla_{\bg_s}\bf_s,\nabla_{\bg_s}f_{\bg_s}\rangle\\
&=&\Delta_{\bg_s}\bf_s-|\nabla_{\bg_s}\bf_s|^2+R_{\bg_s}-\frac{n}{2}+\langle\nabla_{\bg_s}\bf_s,\nabla_{\bg_s}f_{\bg_s}\rangle\\
&=&\Delta_{\bg_s}\left(\bf_s-f_{\bg_s}\right)-\langle\nabla_{\bg_s}\bf_s,\nabla_{\bg_s}\left(\bf_s-f_{\bg_s}\right)\rangle+\Delta_{\bg_s}f_{\bg_s}+R_{\bg_s}-\frac{n}{2}\\
&=&\Delta_{\bg_s}^{f_{\bg_s}}\left(\bf_s-f_{\bg_s}\right)-\left|\nabla_{\bg_s}\left(\bf_s-f_{\bg_s}\right)\right|^2+\tfrac{1}{2}\text{tr }_{\bg_s}\left(\partial_s \bg_s\right),
\end{eqnarray*}
where the norms and inner products are all computed with respect to the evolving metric $\bg_s$. Since $f_g=P(g)$, we have $\partial_s f_{\bg_s}= DP_{\bg_s}\left( \partial_s \bg_s\right)$, where $DP$ is the linearization of the analytic map $P:g\mapsto f_g$. Hence
\begin{equation}\label{eqn for fb-bf}
    \partial_s \left(f_{\bg_s}-\bf_s\right)=\Delta_{\bg_s}^{f_{\bg_s}}\left(f_{\bg_s}-\bf_s\right)+\left|\nabla_{\bg_s}\left(\bf_s-f_{\bg_s}\right)\right|^2 +\left(D P_{\bg_s}-\tfrac{1}{2}\text{tr }_{\bg_s}\right)\left(\partial_s \bg_s\right).
\end{equation}
Let 
\begin{equation}\label{someotherbasicdefinitions}
    v_s:=e^{f_{\bg_s}-\bf_s},\quad
\EE_s := \left(D P_{\bg_s}-\tfrac{1}{2}\text{tr }_{\bg_s}\right)\left(\partial_s \bg_s\right),
\end{equation}
then \eqref{eqn for fb-bf} can be rewritten as
\begin{eqnarray*}
\partial_s v_s=v_s\,\partial_s \left(f_{\bg_s}-\bf_s\right)&=&v_s\,\Delta_{\bg_s}^{f_{\bg_s}}\left(f_{\bg_s}-\bf_s\right)+v_s\,\left|\nabla_{\bg_s}\left(\bf_s-f_{\bg_s}\right)\right|^2 
+ \EE_s v_s
\\
&=& \Delta_{\bg_s}^{f_{\bg_s}} v_s
+ \EE_s v_s.
\end{eqnarray*}

By \eqref{coarseconvergenceoff2}, it is clear that $v_s\to 1$ uniformly. We shall then estimate the rate of this convergence. Let us define 
\begin{equation}\label{somebasicdefinitions}
d\bar \nu_s := (4\pi)^{-n/2}e^{-f_{\bg_s}}\, d\bg_s,\quad
Z(s):=
\left(\int (v_s-1)^{2}
\, d\bar\nu_s\right)^{\frac{1}{2}}.
\end{equation}
Then $d\bar\nu_s$ is a probability measure with
\[
    \partial_s d\bar\nu_s
    = -\EE_s \, d\bar\nu_s.
\]

\begin{Lemma}\label{v is bdd}
There is a positive constant $C_1$ such that for all $s\ge 0$,
\begin{equation}\label{eq: v bdd}
    C_1^{-1}\le v_s\le C_1.
\end{equation}
\end{Lemma}
\begin{proof}
Since the original Ricci flow $g_t$ is of Type I, applying the gaussian upper and lower estimates of \cite[Theorem 3.5]{X17}, in combination with the radius bound
$$\operatorname{diam}_{\bg_s}(M)\le 2\operatorname{diam}_{g_o}(M)\leq C\quad\text{ for all }\quad s\ge 0,$$
we have that, there exists a constant $C'>0$, such that
$$-C'\le \bf_s\le C'\quad\text{ for all }\quad s\ge 0.$$
On the other hand, since $\bg_s$ and $g_o$ are close and consequently $f_{\bg_s}$ and $f_o$ are close, we also have
$$ -C'\le f_{\bg_s}\le C'\quad \text{ for all }\quad s\ge 0.$$
The lemma then follows from the definition of $v_s$.
\end{proof}

\begin{Lemma}
For $Z(s)$ defined in \eqref{somebasicdefinitions}, we have
\[
Z(s)\le C s^{-\theta},
\]
for $s>0,$ where $\theta:=\frac{1-(2-\beta)\alpha}{2\alpha-1}$ is the same constant as in \eqref{bg vs go}.
\end{Lemma}
\begin{proof}
\begin{align}\label{somenonsensecomputation}
      (Z^2)'(s)
    &= \int 2(v_s-1)\partial_s v_s  \, d\bar\nu_s-\int (v_s-1)^2 \EE_s  \, d\bar\nu_s\\\nonumber
    &= \int 2(v_s-1)\left(\Delta_{\bg_s}^{f_{\bg_s}} v_s+\EE_s v_s\right)  \, d\bar\nu_s-\int (v_s-1)^2 \EE_s  \, d\bar\nu_s\\\nonumber
    &=-2\int |\nabla_{\bg_s} (v_s-1)|^2  \, d\bar\nu_s+ \int (v_s^2-1)\EE_s  \, d\bar\nu_s\\\nonumber
    &\le - \int (v_s-1)^2  \, d\bar\nu_s + \sup_M \left(|v_s+1|\cdot|\EE_s|\right)\int |v_s-1| \, d\bar\nu_s\\\nonumber
     &=-Z^2(s)+C\sup_M|\EE_s|\cdot Z(s),
     \end{align}
where we have applied \eqref{Poincare}, the fact that $\int v_sd\bar\nu_s\equiv 1$, and Lemma \ref{v is bdd}. For the $\EE_s$ term, we may estimate using 
\eqref{someotherbasicdefinitions} and Theorem \ref{derivative bdd of P}:
\begin{equation}\label{EstimateofEE}
    \sup_M|\EE_s|\leq C\|\partial_s\bg_s\|_{C^0_{g_o}}\leq -C\tfrac{d}{ds}\left(\mu_{\bg_s}-\mu_{g_o}\right)^{1-(2-\beta)\alpha}.
\end{equation}
Consequently, \eqref{somenonsensecomputation} becomes
\begin{equation}\label{comparizonODE}
    Z'(s)
    \le - \tfrac{1}{2}Z(s)
    - C\tfrac{d}{ds}\left(\mu_{\bg_s}-\mu_{g_o}\right)^{1-(2-\beta)\alpha}.
\end{equation}

Let us define
\[
    \zeta(s):= C_0\left(\mu_{\bg_s}-\mu_{g_o}\right)^{1-(2-\beta)\alpha},
\]
where $C_0$ is some large constant depending on $v_0$ and $g_o$ to be determined. We may first choose $C_0\ge C$, where $C$ is the constant in \eqref{comparizonODE}, so that
\begin{equation}\label{comparizonODE1}
        (Z+\zeta)'(s)
    \le - \tfrac{1}{2}Z(s).
\end{equation}
Recall that \eqref{mu decay} implies that
\begin{equation}
\label{ineq: zeta}
    \zeta(s)
    \le C_0C s^{-\theta}\quad\text{ for all }\quad s>0.
\end{equation}
On the other hand, $\mu_{\bg_s}-\mu_{g_o}$ is strictly positive, for otherwise the Ricci flow in question is trivial, so
we may enlarge $C_0$ such that $Z(0)\le \zeta(0)/2.$

Let us now proceed to estimate $Z(s)$. We fix an arbitrary $s\ge 1$. If $Z(s)\leq \zeta(s)$, then, by \eqref{ineq: zeta}, we are done. Let us assume $Z(s)>\zeta(s)$, and define
\[
    s_1:=\inf\{s'>0: Z\ge \zeta \text{ on the interval }[s',s]\}.
\]
Since $Z(0)\le \zeta(0)/2$, we must have $s_1\in(0,s)$. Hence, $Z(s_1)=\zeta(s_1)$. Then, on the interval $[s_1,s]$, the fact $Z\geq\zeta$ together with \eqref{comparizonODE1} implies
\[
    (Z+\zeta)'\le 
    - \tfrac{1}{2}Z
    \le -\tfrac{1}{4}(Z+\zeta).
\]
Integrating the above inequality from $s_1$ to $s$, we have
\[
(Z+\zeta)(s)
\le (Z+\zeta)(s_1)e^{-\frac{1}{4}(s-s_1)}=2\zeta(s_1)e^{-\frac{1}{4}(s-s_1)}
\le C s_1^{-\theta}e^{-\frac{1}{4}(s-s_1)},
\]
where we have also applied \eqref{ineq: zeta}.
If $s_1\geq \frac{1}{2}s$, then
\[
Z(s)
\le (Z+\zeta)(s)
\le Cs_1^{-\theta}
\le C2^{-\theta} s^{-\theta}.
\]
If $s_1\leq \frac{1}{2}s$, then
\[
    Z(s)
    \le 2\zeta(s_1)e^{-s/8}
    \le Cs^{-\theta};
\]
note that $\zeta$ is uniformly bounded.
\end{proof}

Now we summarize the conclusion in the backward case.

\begin{Proposition}\label{conclusionofthereferentialconvergence}
Let $(M^n,g_t)_{t\in(-\infty,0]}$ be the ancient Ricci flow in the statement of Theorem \ref{Thm_main_1}, where $M$ is a closed manifold. Let $d\mu_t:=(4\pi|t|)e^{-f_t}dg_t$ be the referential conjugate heat flow. Let $f^*_{g_t}$ be the minimizer of $\mu(g_t,|t|)$, where $\mu$ is Perelman's $\mu$-functional. Then for any $\theta'\in(0,\theta)$, there is a positive constant $C$, such that
\[
    \|f_t - f^*_{g_t}\|_{C^0}
    \le C(\log|t|)^{-\theta'}\quad \text{ for all }\quad t\leq -1.
\]
Note that the $C^0$ norm is independent of the choice of the metric.
\end{Proposition}
\begin{proof}
Indeed, we need only to estimate $\left|\bf_s-f_{\bg_s}\right|$. By Lemma \ref{v is bdd}, we have
\begin{align*}
    \left|\bf_s-f_{\bg_s}\right|\leq \left(\sup_{x\in[C_1^{-1},C_1]}\left|(\log x)'\right|\right)\cdot\left|v_s-1\right|\leq C|v_s-1|.
\end{align*}
Hence, we have
\begin{align*}
    \left\|\bf_s-f_{\bg_s}\right\|_{L^2_{g_o}}\leq C\|v_s-1\|_{L^2_{g_o}}\leq C\|v_s-1\|_{L^2_{\bar\nu_s}}=CZ(s)\leq Cs^{-\theta}.
\end{align*}

Furthermore, since, by the standard parabolic derivative estimates and the regularity of $\bg_s$, the higher derivatives of $\bf_s$ and $f_{\bg_s}$ are bounded uniformly in $s$, we may apply the standard Sobolev embedding and interpolation formula (c.f. \cite{Ham82}). Specifically, for any $\theta'\in(0,\theta)$, we may find an integer $N\gg 1$, such that
\begin{align*}
   \left\|\bf_s-f_{\bg_s}\right\|_{C^0}\leq C(\theta') \left\|\bf_s-f_{\bg_s}\right\|^{\theta'/\theta}_{L^2_{g_o}}\cdot \left\|\bf_s-f_{\bg_s}\right\|^{1-\theta'/\theta}_{W^{N,2}_{g_o}}\leq Cs^{-\theta'}.
\end{align*}
The conclusion of the proposition follows from the definition of $\bf_s$ and $f_{\bg_s}$, and the change of variable $s=\log(-t)=\log|t|$.
\end{proof}

\bigskip

\noindent\textbf{(2) Forward case.} The forward case is almost identical to the backward one. We will omit most of the computational details and will focus on the points where it is different from the former case. Defining $v_s$, $\EE_s$ as in (\ref{someotherbasicdefinitions}), and $\bar\nu_s$, $Z(s)$ as in (\ref{somebasicdefinitions}), we have
\begin{gather*}
    \partial_sv_s=-\Delta_{\bg_s}^{f_{\bg_s}}v_s+\EE_sv_s,\\
    \partial_s d\bar\nu_s
    = -\EE_s \, d\bar\nu_s,\\
    \int_M v_sd\bar\nu_s\equiv 1.
\end{gather*}
Then, we may compute as in \eqref{somenonsensecomputation}
\begin{align*}
      (Z^2)'(s)
    &= \int 2(v_s-1)\left(-\Delta_{\bg_s}^{f_{\bg_s}} v_s+\EE_s v_s\right)  \, d\bar\nu_s-\int (v_s-1)^2 \EE_s  \, d\bar\nu_s\\\nonumber
    &=2\int |\nabla_{\bg_s} (v_s-1)|^2  \, d\bar\nu_s+ \int (v_s^2-1)\EE_s  \, d\bar\nu_s\\\nonumber
    &\ge  \int (v_s-1)^2  \, d\bar\nu_s - \sup_M \left(|v_s+1|\cdot|\EE_s|\right)\int |v_s-1| \, d\bar\nu_s\\\nonumber
     &=Z^2(s)-C\sup_M|\EE_s|\cdot Z(s),
     \end{align*}
where we have applied Lemma \ref{Poincare} and \eqref{eq: v bdd}; obviously, the latter formula is also valid in this case due to the gaussian estimates of \cite[Proposition 2.7, Proposition 2.8]{MM15}. Arguing as in (\ref{EstimateofEE}) and using (\ref{forwardconvergencerate}), we have
\begin{equation*}
    \sup_M|\EE_s|\leq C\|\partial_s\bg_s\|_{C^0_{g_o}}\leq -C\tfrac{d}{ds}\left(\mu_{g_o}-\mu_{\bg_s}\right)^{1-(2-\beta)\alpha}.
\end{equation*}
Therefore, we have 
\begin{align*}
    Z'(s)\ge \tfrac{1}{2} Z(s)+C\tfrac{d}{ds}\left(\mu_{g_o}-\mu_{\bg_s}\right)^{1-(2-\beta)\alpha},
\end{align*}
and 
\begin{align*}
    \tfrac{d}{ds}\left(Z(s)-C\left(\mu_{g_o}-\mu_{\bg_s}\right)^{1-(2-\beta)\alpha}\right)\ge \tfrac{1}{2}Z(s)\ge \tfrac{1}{2}\left(Z(s)-C\left(\mu_{g_o}-\mu_{\bg_s}\right)^{1-(2-\beta)\alpha}\right).
\end{align*}

Defining $$\xi(s):=Z(s)-C\left(\mu_{g_o}-\mu_{\bg_s}\right)^{1-(2-\beta)\alpha},$$
we have
\begin{align}\label{anothernonsensedifferentialineq}
    \xi'(s)\geq \tfrac{1}{2}\xi(s)\quad\text{ for all }\quad s\ge 0.
\end{align}
If there is a $s_0\ge 0$ such that $\xi(s_0)>0$, then, integrating \eqref{anothernonsensedifferentialineq} from $s_0$ to $s\in(s_0,\infty)$, we have
$$\xi(s)\ge \xi(s_0)\exp\left(\tfrac{1}{2}(s-s_0)\right)\to \infty\quad\text{ as }\quad s\to \infty;$$
this obviously is a contradiction. 

In conclusion, we have that $\xi(s)\leq 0$ for all $s\in[0,\infty)$, and hence
$$Z(s)\le C\left(\mu_{g_o}-\mu_{\bg_s}\right)^{1-(2-\beta)\alpha}\leq Cs^{-\theta}\quad\text{ for all }\quad s>0,$$
where $\theta:=\frac{1-(2-\beta)\alpha}{2\alpha-1}\in (0,1)$. The rest of the argument is identical to the forward case. We summarize the above results in the following proposition.

\begin{Proposition}
Let $(M^n,g_t)_{t\in[-T,0)}$ be the Type I Ricci flow in the statement of Theorem \ref{Thm_main_2}, where $M$ is a closed manifold. Let $d\mu_t:=(4\pi|t|)e^{-f_t}dg_t$ be the referential conjugate heat flow, which is a singular conjugate heat kernel defined in \cite{MM15}. Let $f^*_{g_t}$ be the minimizer of $\mu(g_t,|t|)$, where $\mu$ is Perelman's $\mu$-functional. Then for any $\theta'\in(0,\theta)$, there is a positive constant $C$, such that
\[
    \|f_t - f^*_{g_t}\|_{C^0}
    \le C(-\log|t|)^{-\theta'}\quad \text{ for all }\quad t\in [-T/2,0).
\]
Note that the $C^0$ norm is independent of the choice of the metric.
\end{Proposition}

\section{\texorpdfstring{$\mathbb{F}$}{F}-distance estimate}

Finally, we prove Theorem \ref{Thm_main_1} and Theorem \ref{Thm_main_2} in this section. The following result due to Bahuaud-Guenther-Isenberg \cite{BGI20} will be applied to convert the closeness between $\bg_s$ and $g_o$ to the closeness between the scaled ancient Ricci flow and the canonical form of the shrinker.

\begin{Theorem}[Theorem A in \cite{BGI20}]
\label{thm: stability}
Let $M^n$ be a closed manifold and $(M,g_0(t))_{t\in [0,\tau_0)}$ be the maximal solution to the Ricci flow with initial metric $g_0(0)=\bar g_0,$ where $\tau_0\in (0,\infty].$
Then for any $\tau\in (0,\tau_0)$, any integer $k\ge 4$, and any $\gamma\in (0,1),$ there are $r>0$ and $C<\infty$ depending on $\bar g_0$ and  $\tau$ such that
if $\bar g_1$ is another smooth metric on $M$ satisfying
\[
    \|\bar g_1 - \bar g_0\|_{C^{k,\gamma}}
    \le r,
\]
then the maximal Ricci flow $g_1(t)$ starting at $\bar g_1$ exist on $[0,\tau]$ and
\[
    \|g_1(t)-g_0(t)\|_{C^{k-2,\gamma}}
    \le C\|\bar g_1 - \bar g_0\|_{C^{k,\gamma}},
\]
for all $t\in [0,\tau].$ Here, all the H\"{o}lder norms are induced from some fixed background metric on $M.$

\end{Theorem}

Note that in their statements in \cite{BGI20}, the authors used the notation $h^{k,\gamma}$ to denote the completion of smooth sections with respect to the H\"{o}lder norm $C^{k,\gamma},$ which is strictly contained in the usual H\"{o}lder space. We shall not make the explicit distinction here as we only need the H\"{o}lder norms.

The proof of Theorem \ref{Thm_main_1} and Theorem \ref{Thm_main_2} are almost identical, we shall only consider the former. Let $(M^n,g_t)_{t\in(-\infty,0]}$ be the Ricci flow in the statement of Theorem \ref{Thm_main_1}. Let $(M^n,g_o,f_o)$ be the Ricci shrinker whose canonical form $(M,g^o_t,f^o_t)_{t\in(-\infty,0)}$ is the unique tangent flow at infinity of $(M^n,g_t)_{t\in(-\infty,0]}$.  Specifically, let $\Phi_t$ be the group of $1$-parameter family of diffeomorphisms generated by $\nabla_{g_o}f_o,$ then we have
$$g^o_t=|t|\Phi^*_{-\log|t|} g_o, \quad f^o_t=\Phi^*_{-\log|t|}f.$$
Let $A>1$ be an arbitrarily fixed constant and $I_A:=[-A,-1/A].$
For any $\lambda\ll A^{-1},$ write
$g^\lambda_t:= \lambda^2 g_{t/\lambda^2}$
and $s_-:= \log(A/\lambda^2).$ Then, by \eqref{backwardconvergencerate}, we have
\[
    \|\Phi_{-\log A}^*\psi_{s_-}^*g^{\lambda}_{-A}
    - g^o_{-A}\|_{C_{g_o}^{k,\gamma}}
    =  \|\Phi_{-\log A}^*(A\bg_{s_-}
    - Ag_o)\|_{C_{g_o}^{k,\gamma}}
    \le C(A)(-\log \lambda)^{-\theta},
\]
where $\psi_s$ is defined in \eqref{backward_modified_RF}. By Theorem \ref{thm: stability}, if $\lambda\leq\overline{\lambda}(A)$, then we have
\begin{equation}\label{stabilitynonsense}
\sup_{t\in I_A} \|\Psi_\lambda^* g^{\lambda}_t
- g^o_t\|_{C_{g_o}^{k-2,\gamma}}
\le 
C(A) (-\log \lambda)^{-\theta},
\end{equation}
where $\Psi_\lambda:= \psi_{s_-}\circ \Phi_{-\log A}$ is a diffeomorphism, and we are using $g_o$ as the fixed background metric.
In the following, we may assume that $\Psi_{\lambda}={\rm id}$ by considering the pullback flow.

Let $d\mu_t:=u_tdg_t:=(4\pi|t|)e^{-f_t}$ be the referential conjugate heat flow in the statement of Theorem \ref{Thm_main_1}, and $d\mu^\lambda_t:=u^\lambda_tdg^\lambda_t:=(4\pi|t|)e^{-f^\lambda_t}dg^\lambda_t$ its scaled version, where $f^\lambda_t=f_{t/\lambda^2}$. Let $d\mu^o_t:=u^o_tdg^o_t:=(4\pi|t|)e^{-f^o_t}dg^o_t$ be the referential conjugate heat kernel of $(M,g^o_t)$. We shall now compare $\mu^\lambda_t$ and $\mu^o_t$.

By (\ref{stabilitynonsense}) (note that $\Psi^*_\lambda$ is taken to be $\operatorname{id}$), we have
\begin{align}\label{nonsensenonsense}
    \Big\||t|^{-1}\Phi^*_{\log|t|}g^\lambda_t-g_o\Big\|_{C^{k-2,\gamma}_{\Phi^*_{\log|t|}g_o}}&=|t|^{-1}\Big\||t|\Phi^*_{-\log|t|}(|t|^{-1}\Phi^*_{\log|t|}g^\lambda_t-g_o)\Big\|_{C_{g_o}^{k-2,\gamma}}
    \\\nonumber
    &\le C(A)(-\log \lambda)^{-\theta}\quad\text{ for all }\quad t\in I_A.
\end{align}
Now we measure the first norm with the metric $g_o$ instead. 
Since $I_A$ is a compact interval, we have
\begin{equation}\label{uniformequivalence}
    C(A)^{-1}g_o\le \Phi^*_{\log|t|}g_o\le C(A)g_o\quad\text{ for all }\quad t\in I_A.
\end{equation}
Let
$h:=|t|^{-1}\Phi^*_{\log|t|}g^\lambda_t-g_o$ and $\tg:=\Phi^*_{\log|t|}g_o$. By \eqref{uniformequivalence}, we have
$$|h|^2_{g_o}=g_o^{ik}g_o^{jl}h_{ij}h_{kl}\le C(A)^2\tg^{ik}\tg^{jl}h_{ij}h_{kl}\le C(A)|h|^2_{\tg}.$$
Hence, by (\ref{nonsensenonsense}), we have \begin{align}\label{nonsensenonsense1}
    \Big\||t|^{-1}\Phi^*_{\log|t|}g^\lambda_t-g_o\Big\|_{C^0_{g_o}}&\le C(A)\Big\||t|^{-1}\Phi^*_{\log|t|}g^\lambda_t-g_o\Big\|_{C^{0}_{\Phi^*_{\log|t|}g_o}}
    \\\nonumber
    &\le C(A)(-\log \lambda)^{-\theta}\quad\text{ for all }\quad t\in I_A.
\end{align}
Furthermore, by the standard Shi's estimates for $g^\lambda_t$ and the fact that $I_A$ is a compact interval, we obviously have
\begin{align}\label{nonsensenonsense2}
    \Big\||t|^{-1}\Phi^*_{\log|t|}g^\lambda_t-g_o\Big\|_{C^l_{g_o}}\le C(A,l)\quad\text{ for all }\quad t\in I_A,
\end{align}
so long as we take $\lambda\le \overline{\lambda}(A)$ for $\overline{\lambda}(A)$ small enough. Applying the standard interpolation formula with (\ref{nonsensenonsense1}) and (\ref{nonsensenonsense2}), we have that, for any $\theta'\in(0,\theta)$,
$$\Big\||t|^{-1}\Phi^*_{\log|t|}g^\lambda_t-g_o\Big\|_{C^{k,\gamma}_{g_o}}\le C(k,\theta')\Big\||t|^{-1}\Phi^*_{\log|t|}g^\lambda_t-g_o\Big\|^{\theta'/\theta}_{C^{0}_{g_o}}\le C(A,\theta')(-\log \lambda)^{-\theta'}\quad\text{ for all }\quad t\in I_A.$$
Taking $\overline{\lambda}(A)>0$ to be small enough and letting $\lambda\leq \overline{\lambda}(A)$, we may apply Theorem \ref{derivative bdd of P} to obtain
$$\left\| f_{|t|^{-1}\Phi^*_{\log|t|}g^\lambda_t}-f_o\right\|_{C^0}\leq C(A,\theta') (-\log \lambda)^{-\theta'}\quad\text{ for all }\quad t\in I_A,$$
and hence 
\begin{equation}\label{Theclosenessofanothernonsense}
    \left\|f^*_{g^\lambda_t}-f^o_t\right\|_{C^0}\leq C(A,\theta') (-\log \lambda)^{-\theta'}\quad\text{ for all }\quad t\in I_A,
\end{equation}
where $f^*_{g_t^\lambda}$ is the minimizer of $\mu(g_t^\lambda,|t|)$. Note that the $C^0$ norm is independent of the choice of the metric, and $f_o=f_{|t|^{-1}\Phi^*_{\log|t|}g_t^o}$. Combining \eqref{Theclosenessofanothernonsense} with Proposition \ref{conclusionofthereferentialconvergence}, we have
$$\|f^\lambda_t-f^o_t\|_{C^0}\le C(A,\theta') (-\log \lambda)^{-\theta'}\quad\text{ for all }\quad t\in I_A,$$
and consequently 
\begin{equation}\label{lastofalltheclosenessbetweenreferentialchks}
    \|u^\lambda_t-u^o_t\|_{C^0}\le C(A,\theta') (-\log \lambda)^{-\theta'}\quad\text{ for all }\quad t\in I_A.
\end{equation}
Henceforth, we fix a $\theta'\in(0,\theta)$.

Finally, we are ready to apply Corollary \ref{prop: F dist} with \eqref{stabilitynonsense} and \eqref{lastofalltheclosenessbetweenreferentialchks} to conclude the proof of Theorem \ref{Thm_main_1}. Define $\varepsilon:=C(A,\theta')(-\log\lambda)^{-\theta'}$. Let  $\mathfrak{C}=(Z_t,(\phi^\lambda_t, \phi^o_t))_{t\in I_A}$ be the correspondence between $(M,g^\lambda_t,\mu^\lambda_t)_{t\in I_A}$ and $(M,g^o_t,\mu^o_t)_{t\in I_A}$ constructed at the beginning of Section 3.
Fix $t\in I_A$ and let $F:Z_t\to \mathbb{R}$ be any $1$-Lipschitz function. Write $F_\lambda=F\circ{\phi^\lambda_t}$ and $F_o=F\circ{\phi^o_t}$. Since $F$ is $1$-Lipschitz, for any $x\in M,$ we have
\[
    |F_\lambda(x)-F_o(x)|
    = |F(\phi^\lambda_t(x))
    - F(\phi^o_t(x))|
    \le \dist^{Z_t}(\phi_t^\lambda(x),
    \phi^o_t(x))
    = \varepsilon.
\]
By adding a constant to $F$, we may assume that $F_\lambda$ vanishes at one point on $M$. 
It follows that
\begin{align*}
    &\int_{Z_t} F\, d \phi^{\lambda}_{t*}\mu^{\lambda}_t
    - \int_{Z_t} F\, d \phi^{o}_{t*}\mu^{o}_t
    =
    \int_{M} F_\lambda \,d\mu^{\lambda}_t
    -\int_{M} F_o \, d\mu^{o}_t\\
   \le & \ 
    \int_{M} F_\lambda \,d\mu^{\lambda}_t
    -\int_{M} F_\lambda \, d\mu^{o}_t + \varepsilon
    = \int_{M} F_\lambda u^\lambda_t\,dg^{\lambda}_t
    -\int_{M} F_\lambda \, u^o_t \,dg^o_{t} + \varepsilon\\
    \le&\ 
    \int_M F_\lambda(u^{\lambda}_t - u^o_t)\,dg^{\lambda}_t
    + \int_{M} F_\lambda \, u^o_t \left(dg_{t}^\lambda-dg^o_{t}\right)
    + \varepsilon\\
    \le&\ 
    C(A) (-\log \lambda)^{-\theta'}.
\end{align*}
Note that the bound of $F_\lambda$ comes from the fact that $F_\lambda$ attains $0$ at some point on $M$, that $F_\lambda$ is $1$-Lipschitz, and that $(M,g^\lambda_t)$ has uniformly bounded diameter for $t\in I_A$. Since $F$ is arbitrary, we have
\[
    \sup_{t\in I_A} \dist^{Z_t}_{{\rm W}_1}
    \left(
    \phi^{\lambda}_{t*}\mu^{\lambda}_t,
    \phi^{o}_{t*}\mu^{o}_t
    \right)
    \le  C(A) (-\log \lambda)^{-\theta'},
\]
and Theorem \ref{Thm_main_1} follows from Corollary \ref{prop: F dist}.

\bibliographystyle{amsalpha}

\newcommand{\etalchar}[1]{$^{#1}$}
\providecommand{\bysame}{\leavevmode\hbox to3em{\hrulefill}\thinspace}
\providecommand{\MR}{\relax\ifhmode\unskip\space\fi MR }
\providecommand{\MRhref}[2]{%
  \href{http://www.ams.org/mathscinet-getitem?mr=#1}{#2}
}
\providecommand{\href}[2]{#2}

\bigskip
\bigskip


\noindent Department of Mathematics, University of California, San Diego, CA 92093, USA
\\ E-mail address: \verb"pachan@ucsd.edu "
\\

\noindent Department of Mathematics, University of California, San Diego, CA 92093, USA
\\ E-mail address: \verb"zim022@ucsd.edu"
\\

\noindent School of Mathematics, University of Minnesota, Twin Cities, MN 55414, USA
\\ E-mail address: \verb"zhan7298@umn.edu"

\end{document}